\documentclass[12pt]{amsart}

\usepackage{graphicx}
\usepackage{geometry}
\usepackage{amsthm,amssymb,mathrsfs}
\usepackage{amsmath,amscd}
\usepackage{graphicx}
\usepackage{xcolor}
\usepackage{pinlabel}
\usepackage{fourier}
\usepackage[breaklinks,colorlinks,citecolor=blue,linkcolor=blue,
 urlcolor=teal]{hyperref}
\usepackage[all]{xy}

\newtheorem{theorem}{Theorem}[section]
\newtheorem{lemma}[theorem]{Lemma}

\theoremstyle{definition}
\newtheorem{definition}[theorem]{Definition}

\theoremstyle{remark}
\newtheorem{remark}[theorem]{Remark}

\newtheorem{corollary}[theorem]{Corollary}

\numberwithin{equation}{section}

\DeclareMathOperator{\im}{im}
\DeclareMathOperator{\coker}{coker}
\DeclareMathOperator{\Res}{Res}
\DeclareMathOperator{\Ind}{Ind}

\begin{document}
\title{Stability for Representations of Hecke Algebras of type $A$}

\author{Kun Wang$^1$, Haitao Ma$^2$ and Zhu-Jun Zheng$^1$}
\address{1. Mathematics Department of South China University\\
Guangzhou, P. R. China, 510641}
\email{zhengzj@scut.edu.cn}
\address{2. College Of Science, Harbin Engineering University, Harbin, China, 150001}
\email{hmamath@163.com}

\thanks{The authors was supported in part by NSFC Grant \#11571119 and 11475178.}
\thanks{zhengzj@scut.edu.cn}


\date{}

\maketitle

\begin{abstract}
In this paper we introduce the notion of the stability of a sequence of modules over Hecke algebras. We prove that a finitely generated consistent sequence associated with Hecke algebras is representation stable.

\end{abstract}

\maketitle

\section{Introduction}
\noindent In order to compare the representations $V_n$ of a natural family groups $G_n$, such as symmetric group $\mathfrak{S}_n,$ $GL_n(\mathbb{Q}),$ $SL_n(\mathbb{Q})$, Church and Farb introduced the theory of \emph{representation stability} in their paper \cite{4}. The idea is mainly to study the consistent sequence of representations $V_n$ of groups $G_n$. Let $G_n$ be one of the families $\mathfrak{S}_n,$ $GL_n(\mathbb{Q}),$ $SL_n(\mathbb{Q})$. A sequence of $G_n$-representations
$$V_0\overset{\phi_0}\longrightarrow V_1\overset{\phi_1}\longrightarrow V_2\overset{\phi_2}\longrightarrow\cdots$$
is called \emph{consistent sequence} if for all $n\geq0$ and $g\in G_n$, the following diagram
 $$\CD
   V_n @>\phi_n>> V_{n+1} \\
   @V g VV @V g VV  \\
   V_n @>\phi_n>> V_{n+1}
 \endCD
 $$
is commutative, where each $\phi_n$ is a linear map. A consistent sequence is called representation stable, if for sufficiently large $n$, it satisfies:
\begin{itemize}
  \item \textbf{Injectivity}: The map $\phi_n$ is injective.
  \item \textbf{Surjectivity}: The span of the image of $\phi_n$ is equal to $V_{n+1}$.
  \item \textbf{Multiplicities}: Decompose $V_n$ into irreducible representations as
$$V_n=\bigoplus\limits_\lambda c_{\lambda,n}V(\lambda)_n$$
with multiplicities $0\leq c_{\lambda,n}\leq\infty$. For each partition $\lambda$, $c_{\lambda,n}$ is independent of $n$.
\end{itemize}

In order to study the consistent sequence of representations of the symmetric groups $\mathfrak{S}_n$ over a commutative ring $R$, Church, Ellenberg and Farb introduced and developed the theory of $FI$-modules in their paper\cite{3}, where $FI$ is the category of finite sets and injective maps. An FI-module over a commutative ring $R$ is a functor from the category $FI$ to the category of $R$-modules. Thus a sequence of representations of symmetric groups is encoded by a single object $FI$-module. An important result is that a finitely generated $FI$-module over a field of characteristic 0 can give rise a consistent sequence which turns out to be representation stable\cite{3}.


In \cite{6}, Gan and Watterlond researched $VI$-modules, where $VI$ is the category of finite dimensional vector spaces and injective linear maps. They proved a result about the representation stability for the family of finite general linear groups $GL_n(\mathbb{F}_q)$. That is, a $VI$-module over an algebraically closed field of characteristic zero is finitely generated if and only if the consistent sequence obtained from $V$ is representation stable and $\dim(V_n)<\infty$ for each $n$.

In this paper, we imitate the idea of Church, Ellenberg and Farb \cite{3} to introduce a similar idea for the family of the Iwahori-Hecke algebras. We prove that a finitely generated consistent sequence associated with Hecke algebras is representation stable.
\section{preliminaries}

Let's first list some basic results about Hecke algebras and some notations.

Let $k$ be a field of characteristic 0 and $q\in k$ which is not a root of unity. The Iwahori-Hecke algebra(or simply the Hecke algebra) $\mathscr{H}_n=\mathscr{H}_{k,q}(\mathfrak{S}_n)$ of the symmetric group $\mathfrak{S}_n$ is the unital associative $k$-algebra with generators $\{T_{s_1}, T_{s_2}, \ldots, T_{s_{n-1}}\}$ and relations

$(T_{s_i}-q)(T_{s_i}+1)=0$, for $i=1,2,\ldots,n-1$;

$T_{s_i}T_{s_j}=T_{s_j}T_{s_i}$, for $1\leq i<j-1\leq n-2$;

$T_{s_i}T_{s_{i+1}}T_{s_i}=T_{s_{i+1}}T_{s_i}T_{s_{i+1}}$, for $i=1,2,\ldots,n-2$.

We shall denote $\mathscr{H}_{0}=\mathscr{H}_{1}=k.$

Suppose $\sigma\in\mathfrak{S}_n$ and let $\sigma=s_{i_1}\cdots s_{i_l}$ be a \emph{reduced expression}, which means
$$l=\min\{m|\sigma=s_{i_1}\cdots s_{i_m}\}$$
and the number $l$ is denoted by $l(\sigma)$. Denote $T_\sigma:=T_{s_{i_1}}\cdots T_{s_{i_l}}$ which is an element of the Hecke algebra $\mathscr{H}_n$. Then $\mathscr{H}_n$ is free with $k$-basis $\{T_\sigma|\sigma\in\mathfrak{S}_n\}$. If $\sigma_1\in\mathfrak{S}_n, \sigma_2\in\mathfrak{S}_n$ and $l(\sigma_1\sigma_2)=l(\sigma_1)+l(\sigma_2)$, then $T_{\sigma_1\sigma_2}=T_{\sigma_1}T_{\sigma_2}$.

Let $n\in\mathbb{N}$. A sequence of non-negative integers $\lambda=(\lambda_1, \lambda_2, \cdots, \lambda_m)$ which satisfies $\lambda_1+\lambda_2+\cdots+ \lambda_m=n$ is called a composition of $n$. A composition $\lambda$ is called a \emph{partition} of $n$ when $\lambda_1\geq\lambda_2\geq\ldots\geq\lambda_m\geq0$. Composition(partition) is generally denoted by $\lambda\models n$($\lambda\vdash n$). For each composition $\lambda$ of $n$, the standard Young subgroup $\mathfrak{S}_{\lambda}$ of $\mathfrak{S}_n$ is defined by $S_{\{1,2, \ldots, \lambda_1\}}\times S_{\{\lambda_1+1,\lambda_1+2, \ldots, \lambda_1+\lambda_2\}}\times\cdots\times S_{\{\lambda_1+\lambda_2+\cdots+\lambda_{n-1}+1, \ldots, n\}}$. The subalgebra $\mathscr{H}_{\lambda}$ of $\mathscr{H}_n$ generated by $\{T_{\sigma}|\sigma\in \mathfrak{S}_{\lambda}\}$ is $k$-free with basis $\{T_{\sigma}|\sigma\in \mathfrak{S}_{\lambda}\}$. This subalgebra is called Young subalgebra.

Let $\lambda=(\lambda_1, \lambda_2,\ldots)$ be a partition. For any $n\geq|\lambda|+\lambda_1$, the \emph{padded partition} is defined as:
$$\lambda[n]:=(n-|\lambda|, \lambda_1, \lambda_2,\ldots),$$
where $|\lambda|=\sum\limits_{i=1}^\infty\lambda_i$.

In the case of the Hecke algebra $\mathscr{H}_n$ that we have assumed, the simple modules of Hecke algebra can be labelled (up to isomorphism) by the set of all partitions of $n$. For a partition $\lambda$ and $n\geq|\lambda|+\lambda_1$, we denote $S^\lambda$ the corresponding simple $\mathscr{H}_{|\lambda|}$-module and set $S(\lambda)_n:=S^{\lambda[n]}$.

\begin{definition}
A sequence $V=(V_n,\phi_n)_{n\in \mathbb{N}}$ is called
\emph{consistent sequence} or \emph{$FI_{\mathscr{H}}$-module} if it satisfies: for all $h\in \mathscr{H}_{n}$, the diagram
 $$\CD
   V_n @>\phi_n>> V_{n+1} \\
   @V h VV @V \tau(h) VV  \\
   V_n @>\phi_n>> V_{n+1}
 \endCD
 $$
 is commutative, where $\tau: \mathscr{H}_{n}\to \mathscr{H}_{n+1}$ is a $k$-algebra map defined by $\tau(T_{s_i})=T_{s_i}$ for $1\leq i\leq n-1$ and $V_n$ is an $\mathscr{H}_{n}$-module and $\phi_n: V_n\to V_{n+1}$ is a $k$-linear map, .
\end{definition}
Let $V=(V_n, \phi_n)$ and $W=(W_n, \psi_n)$ be two $FI_\mathscr{H}$-modules. A \emph{morphism} $f=(f_n)$ from $V$ to $W$ is a sequence of homomorphisms $f_n: V_n\to W_n$ for all $n\geq 0$ making the following diagram:
$$\CD
  V_n @>\phi_n>> V_{n+1} \\
  @V f_n VV @V f_{n+1} VV  \\
  W_n @>\psi_n>> W_{n+1}
\endCD
$$commutes.

\begin{remark}
$FI_\mathscr{H}$-modules together with the above morphisms is an abelian category. Notions such as kernel, cokernel, injection, and surjection are defined by pointwise. The direct sum and tensor product also make sense. For example, let $f=(f_n): V\to W$ be a morphism. $(\ker f)_n$ is defined by $\ker(f_n)$ and the map from $(\ker f)_n$ to $(\ker f)_{n+1}$ is the restriction of $\phi_n$. The tensor product $V\bigotimes W$ of two $FI_\mathscr{H}$-modules $V=(V_n,\phi_n)$ and $W=(W_n,\psi_n)$ is defined by $(V\bigotimes W)_n=V_n\bigotimes_kW_n$ and the map from $(V\bigotimes W)_n$ to $(V\bigotimes W)_{n+1}$ is defined by $\phi_n\otimes\psi_n$.
\end{remark}

\begin{definition}
Let $W=(W_n)_{n\geq0}$, where $W_n$ is an $\mathscr{H}_{n}$-module. Denote $\lambda_{m,n}$ the composition $(m,n-m)$ for $m\leq n$. Then $W_m$ will be an $\mathscr{H}_{\lambda_{m,n}}$-module under the action $T_{\sigma}w=T_{\sigma_1\sigma_2}w=T_{\sigma_1}w$, where $\sigma=\sigma_1\sigma_2$ for $\sigma_1\in S_{\{1,2,\ldots,m\}}$ and $\sigma_2\in S_{\{m+1,\ldots,n-m\}}$ and $w\in W_m$. We define a sequence $M(W)$:
For $n\geq0$,
$$M(W)_n:=\bigoplus_{m\leq n}(\mathscr{H}_n\bigotimes_{\mathscr{H}_{\lambda_{m,n}}}W_m),$$
with the map $\phi_n: M(W)_n\to W(W)_{n+1}$ determined by $h\otimes w_m\mapsto \tau(h)\otimes w_m$, where $h\in\mathscr{H}_n$ and $w_m\in W_m$.

\end{definition}

It is easy to check that $M(W)$ is an $FI_\mathscr{H}$-module. If $W=(W_n)_{n\geq0}$ satisfies $W_m=\mathscr{H}_m$ for some $m$ and $W_n=0$ for $n\neq m$, then we denote $M(W)$ by $M(m)$.

Let $V=(V_n,\phi_n)$ be an $FI_\mathscr{H}$-module. Denote $\phi_{n-1,m}=\phi_{n-1}\phi_{n-2}\cdots\phi_m: V_m\to V_n$ for $0\leq m\leq n-2$ and $\phi_{n-1,n-1}=\phi_{n-1}$. For a disjoint union $S=\bigsqcup\limits_{n=0}^{\infty}S_n$, where $S_n\subseteq V_n$. We define $span(S)_n$ to be the submodule of $V_n$ generated by $\big(\bigcup\limits_{i=0}^{n-1}\phi_{n-1,i}(S_i)\big)\cup S_n$. And the map $\widetilde{\phi}_n: span(S)_n\to span(S)_{n+1}$ is the restriction of $\phi_n$. This makes $span(S)$ to be an $FI_\mathscr{H}$-module.
\begin{definition}
We say that an $FI_\mathscr{H}$-module $V$ is generated in degree $\leq n$ if $V$ is generated by elements of $V_m$ for $m\leq n$. That is, $span(\bigsqcup_{m\leq n}V_m)=V$.
\end{definition}

\begin{definition}
We say that an $FI_\mathscr{H}$-module $V$ is finitely generated if there is a finite set $S$ of the disjoint union $\bigsqcup V_n$ such that $span(S)=V$.
\end{definition}

\begin{definition}
\textbf{(Representation stability of consistent sequence)}
Let $(V_n,\phi_n)_{n\in \mathbb{N}}$ be an $FI_\mathscr{H}$-module. This sequence is called \emph{uniformly representation stable} if there exists an integer $N$ such that for each with $n\geq N$, the following conditions hold:
\begin{itemize}
  \item \textbf{Injectivity}: $\phi_n: V_n\to V_{n+1}$ is injective.
  \item \textbf{Surjectivity}: $V_{n+1}$ is generated by the image of $\phi_n$.
  \item \textbf{Multiplicities}: $V_n=\bigoplus_{\lambda} c_{\lambda,n}S(\lambda)_n$, where the multiplicities $c_{\lambda,n}$ is independent of $n$.
\end{itemize}
\end{definition}
This definition of representation stability is in the sense of Church and Farb\cite{4}. They introduced the idea of representation stability for a sequence of representations of groups. An important result in \cite{3} is that a consistent sequence of $\mathfrak{S}_n$-representations is representation stable can be converted to a finite generation property for an $FI$-module.


\section{Stability degree}
In this section, we define the notions of stability degree, injective degree and surjective degree for an $FI_\mathscr{H}$-module. For $W=(W_n)_{n\geq0}$, we will compute these numbers for the $FI_\mathscr{H}$-module $M(W)$. The following lemma shows that $M(W)$ is similar to the "free object".
\begin{lemma}\label{Free}
If an $FI_\mathscr{H}$-module $V=(V_n,\phi_n)$ is generated in degree $\leq d$, then there exists an epimorphism from $\bigoplus\limits_{i=0}^dM(V_i)$ to $V$.
\end{lemma}
\begin{proof}
For every $0\leq n\leq d$, $V_n$ is a summand of $(\bigoplus\limits_{i=0}^dM(V_i))_n$ since $M(V_n)_n=V_n$ by definition. So the functions from $\big(\bigoplus\limits_{i=0}^dM(V_i)\big)_n$ to $V_n$ are defined by projections. Now let $n>d$, then $\big(\bigoplus\limits_{i=0}^dM(V_i)\big)_n=\bigoplus\limits_{i=0}^d\big(\mathscr{H}_n\bigotimes_{\mathscr{H}_{\lambda_i}}V_i\big)$. If $v_i\in V_i$ for $0\leq i\leq d$, then $f_n(1\otimes v_i)=\phi_{n-1}\cdots\phi_{i+1}\phi_{i}(v_i)$. The fact $\phi_n f_n=f_{n+1}$ implies that $f=(f_n)_{n\geq0}$ is a morphism. From the construction of $f$ and $V$ is generated in degree $\leq d$, we know $f$ is an epimorphism.
\end{proof}

Given $a\geq0$ and an $FI_\mathscr{H}$-module $V=(V_n,\phi_n)$, we define a sequence $\Phi_a(V)$ as follows. Denote by $Q_n$ the subspace of $V_{a+n}$ spanned by $\{T_{\sigma}v-v|\sigma\in S_{\{a+1,a+2, \ldots , a+n\}}v\in V_{a+n}\}$ for $n\geq 0$ and define $\Phi_a(V)_n:=V_{a+n}/Q_n$. Define the map $T: \Phi_a(V)_n\to\Phi_a(V)_{n+1}$ by $[v]\mapsto[\phi_{a+n}(v)]$. The map $T$ is well-defined since for $T_{\sigma}v-v\in Q_n$, $\phi_{a+n}(T_{\sigma}v-v)=T_{\sigma}\phi_{a+n}(v)-\phi_{a+n}(v)\in Q_{n+1}$.
\begin{definition}
If there exists a natural number $s$ such that for all $a\geq0$, the map defined above is an isomorphism for $n\geq s$, then $V$ is called to have \emph{stability degree} $s$. If the map $T$ is injective or surjective, then $V$ is called to have injective degree $s$ or surjective degree $s$. These facts are denoted by stab-deg$(V)$\big(or inj-deg$(V)$, sur-deg$(V)$\big)$=s$.

\end{definition}

The following lemma is about double cosets of symmetric group of two Young subgroups.
\begin{lemma}\label{Double coset}
Given $a\geq0$. For $n\geq0$ and $m\leq a+n$, let $\lambda_n=(m,a+n-m)$ and $\mu_{n}=(1,1,\ldots,1,n)$ be two compositions of $a+n$. Denote $\mathscr{D}_{\mu_n,\lambda_n}$ the representative elements of minimal length of the double cosets $\mathfrak{S}_{\mu_n}\backslash\mathfrak{S}_{a+n}/\mathfrak{S}_{\lambda_n}$. Then $\mathscr{D}_{\mu_n,\lambda_n}\subseteq \mathscr{D}_{\mu_{n+1},\lambda_{n+1}}$ and if $n\geq m$, the set $\mathscr{D}_{\mu_n,\lambda_n}$ is independent of $n$.

\end{lemma}

\begin{proof}
By Lemma 1.7 in\cite{1}, there is a bijection from the set of row-standard tableaux in $\mathscr{T}(\lambda_n,\mu_n)$ to $\mathscr{D}_{\mu_n,\lambda_n}$. These row-stantard $\lambda_n$-tableaux of type $\mu_n$ are determined by the subset of $\{1,2,\ldots,a\}$ which is in the first row of $\lambda_n$. Denote the number of the set of row-standard tableaux in $\mathscr{T}(\lambda_n,\mu_n)$ by $\mathscr{T}_n$. From the constructions of $\lambda_n$ and $\mu_n$, we have $\mathscr{T}_n\leq\mathscr{T}_{n+1}$. If $\sigma\in\mathscr{D}_{\mu_n,\lambda_n}$ is the element of minimal length in a double coset $\mathfrak{S}_{\mu_n}\sigma\mathfrak{S}_{\lambda_n}$, then $\sigma$'s length is also minimal in the double coset $\mathfrak{S}_{\mu_{n+1}}\sigma\mathfrak{S}_{\lambda_{n+1}}$. We obtain that $\sigma\in\mathscr{D}_{\mu_n,\lambda_n}$ implies $\sigma\in\mathscr{D}_{\mu_{n+1},\lambda_{n+1}}$. If $n\geq m$, the number $\mathscr{T}_n$ is independent of $n$. So we complete the proof of the lemma.

\end{proof}

\begin{lemma}Assume $W=(W_n)_{n\geq0}$, where $W_n$ is an $\mathscr{H}_n$-module for $n\geq0$. Then $M(W)$ has injective degree 0; if $W_i=0$ for all $i>m$, then we have surj-deg$\big(M(W)\big)\leq m$.
$M(m)$ has injective degree $0$ and surjective degree $m$.
\end{lemma}
\begin{proof}
For $\alpha\in\mathfrak{S}_{a+n}$, there exist $\sigma\in\mathscr{D}_{\mu_n,\lambda_n}$, $\beta\in\mathfrak{S}_{\mu_n}$ and $\gamma\in\mathfrak{S}_{\lambda_n}$ such that $\alpha=\beta\sigma\gamma$. According to the lemma 1.6 in \cite{1}, these elements can be chosen to satisfy $l(\alpha)=l(\beta)+l(\sigma)+l(\gamma)$. So we have $\mathscr{H}_{a+n}=\bigoplus\limits_{\sigma\in\mathscr{D}_{\mu_n,\lambda_n}}(\mathscr{H}_{\mu_n}T_\sigma\mathscr{H}_{\lambda_n})$.

Then by definition, $M(W)_{a+n}$ can be decomposed as
\begin{align}
M(W)_{a+n} & =\bigoplus_{m\leq a+n}(\mathscr{H}_{a+n}\bigotimes_{\mathscr{H}_{\lambda_n}}W_m)\\
 & =\bigoplus_{m\leq a+n}\big(\bigoplus_{\sigma\in\mathscr{D}_{\mu_n,\lambda_n}}(\mathscr{H}_{\mu_n}T_{\sigma}\bigotimes_{\mathscr{H}_{\lambda_n}}W_m)\big).
\end{align}
The map $\phi_{a+n}: M(W)_{a+n}\to M(W)_{a+n+1}$ is the direct sum of the maps: For $m\leq a+n, \sigma\in\mathscr{D}_{\mu_n,\lambda_n}$,
$\mathscr{H}_{\mu_n}T_{\sigma}\bigotimes_{\mathscr{H}_{\lambda_n}}W_m\to\mathscr{H}_{\mu_{n+1}}T_{\sigma}\bigotimes_{\mathscr{H}_{\lambda_{n+1}}}W_m$ defined by $hT_{\sigma}\otimes w\mapsto\tau(h)T_{\sigma}\otimes w$, where $h\in\mathscr{H}_{\mu_n}$ and $w\in W_m$.
And so
 \begin{align}
 \Phi_a\big(M(W)\big)_n & =M(W)_{a+n}/Q_n\\
  & \cong \bigoplus_{m\leq a+n}\Big(\bigoplus_{\sigma\in\mathscr{D}_{\mu_n,\lambda_n}}\big((\mathscr{H}_{\mu_n}T_{\sigma}\bigotimes_{\mathscr{H}_{\lambda_n}}W_m)/Q_n\big)\Big)\\
  & \cong \bigoplus_{m\leq a+n}\big(\bigoplus_{\sigma\in\mathscr{D}_{\mu_n,\lambda_n}}(T_{\sigma}\bigotimes_{\mathscr{H}_{\lambda_n}}W_m)\big).
 \end{align}
The map $T: \Phi_a(M(W))_n\to\Phi_a(M(W))_{n+1}$ restricted in one factor in the above decomposition is the identity. So we have obtained that the inj-deg$(M(W))=0$. If $W_i=0$ for $i>m$ and $n\geq m$, then $\mathscr{H}_{\mu_n}T_{\sigma}\bigotimes_{\mathscr{H}_{\lambda_n}}W_i=0$. So $T$ is surjective for all $n\geq m$. We have proved the proposition.

\end{proof}
From the above lemma we have:
\begin{corollary}\label{corollary}
The surjective degree of a consistent sequence is less than its generated degree.
\end{corollary}

\section{Weight}
Assume that $q$ is not a root of unity, then the Hecke algebra $\mathscr{H}_n$ is semisimple and all the non-isomorphic simple $\mathscr{H}_n$-modules are indexed by all partitions $\lambda\vdash n$. For each partition $\lambda$, the corresponding irreducible module is denoted by $S^{\lambda}$. By the corollary 6.2 in \cite{2}, the branching rule of Hecke algebra is the same as the classical branching rule for $\mathfrak{S}_n$-representations. The results about the branching rule for $\mathfrak{S}_n$-representations which stated in \cite{3} are also true for the modules of Hecke algebra. For $\lambda\vdash n$ and $m\in\mathbb{N}$, write $\lambda\rightarrow\mu$ if $\mu$ is a partition of $n+m$ such that $[\mu]$ is obtained from $[\lambda]$ by adding $m$ boxes from different columns. We write that lemma of the branching rule in the following:

\begin{lemma}(\textbf{The branching rule})\cite{2}
Let $\lambda$ be a partition of $n$. Then
\begin{enumerate}
  \item $\Ind_{\mathscr{H}_n\bigotimes\mathscr{H}_m}^{\mathscr{H}_{n+m}}S^{\lambda}\cong\bigoplus_{\lambda\rightarrow\mu}S^\mu$;
  \item $(\Res^{\mathscr{H}_n}_{\mathscr{H}_{n-m}\bigotimes\mathscr{H}_m}S^{\lambda})/Q_m\cong\bigoplus_{\mu\rightarrow\lambda}S^\mu$.
\end{enumerate}

\end{lemma}


\begin{lemma}\label{branching lemma}\cite{3}
Let $\lambda$ be a partition of $n$ and $a\leq n$.
\begin{enumerate}
  \item $S(\lambda)_n/Q_{n-a}=0 \Longleftrightarrow n-a>n-|\lambda|$.
  \item $S(\lambda)_n/Q_{n-|\lambda|}\cong S^\lambda$.
  \item If $n\geq a+|\lambda|$, then $S(\lambda)_n/Q_{n-a}$ is independent of $n$.
  \item Let $V$ be a consistent sequence with weight less or equal to $d$. If $W_n/Q_{n-d}=0$ for any subquotient $W_n$ of $V_n$, then $W_n=0$.
\end{enumerate}

\end{lemma}

\begin{definition}
Let $V$ be an $FI_\mathscr{H}$-module. We say a partition $\lambda$ occurs in $V$, if there exists $n\geq0$, $S(\lambda)_n$ occurs in the $\mathscr{H}_{n}$-module $V_n$. The weight of $V$ is defined the maximum of $|\lambda|$ which $\lambda$ occurs in $V$.
\end{definition}

\begin{lemma}
For any partition $\lambda\vdash m$, the $FI_\mathscr{H}$-module $M(S^{\lambda})$ has weight $m$.
\end{lemma}
\begin{proof}
By the definition of $M(S^\lambda)$, $M(S^\lambda)_n=\mathscr{H}_n\bigotimes\limits_{\mathscr{H}_{m,n-m}} S^\lambda=\Ind S^\lambda$. So for any $n\geq m+\lambda_1$, adding $n-m$ boxes in different columns we know the partition $\lambda[n]$ occurs in $M(S^\lambda)$. So the weight of $M(S^\lambda)$ is at least $m$. On the other hand, for any partition $\mu$ that occurs in $M(S^\lambda)_n$ satisfies $\mu_1\geq n-m$. If we write $\mu=\nu[n]$ for some partition $\nu$, then we have $|\nu|=n-\mu_1\leq m$. Then the weight of $M(S^\lambda)$ is at most $m$. We thus complete the proof of this lemma.
\end{proof}
If $V$ is generated in degree $\leq d$, then the weight of $V$ is less or equal to $d$ by the above lemma.

\begin{lemma}
For any partition $\lambda=(\lambda_1,\lambda_2,\cdots)$, the consistent sequence $M(S^{\lambda})$ has stability degree $\lambda_1$.
\end{lemma}
\begin{proof}
From our previous results, we know the stability degree of $M(S_\lambda)$ is less than or equal to $|\lambda|$ and we have:
\begin{equation}\label{decomposition}
  \Phi_aM(S_\lambda)_n\cong \bigoplus_{\sigma\in\mathscr{D}_{\mu_n,\lambda_n}}\big((\mathscr{H}_{\mu_n}T_{\sigma}\bigotimes_{\mathscr{H}_{\lambda_n}}S^\lambda)/Q_n\big),
\end{equation}
where $\mu_n=(1,1,\cdots,1,n)$ and $\lambda_n=(|\lambda|,a+n-|\lambda|)$.

By Lemma\ref{Double coset} and an easily computing, for $\sigma\in\mathscr{D}_{\mu_n,\lambda_n}$ there is some $k\leq |\lambda|$ such that:
$$\sigma^{-1}S_{\mu_n}\sigma\cap S_{\{1,2,\ldots,|\lambda|\}}=S_{\{|\lambda|-k+1,\ldots,|\lambda|\}}.$$
For any $w\in S_{\{|\lambda|-k+1,\cdots,|\lambda|\}}$, there exists $u\in S_{\mu_n}$ such that $w=\sigma^{-1}u\sigma\in\sigma^{-1}S_{\mu_n}\sigma\cap S_{\lambda_n}$. Since $\sigma\in\mathscr{D}_{\mu_n,\lambda_n}$, so $l(\sigma w)=l(\sigma)+l(w)$ and $l(u\sigma)=l(u)+l(\sigma)$. Assume $x$ is an element in $S^\lambda$, then we have:
$$[T_\sigma\otimes T_wx]=[T_\sigma(T_\sigma^{-1}T_uT_\sigma)\otimes x]=[T_uT_\sigma\otimes x]=[T_\sigma\otimes x].$$
By Lemma\ref{branching lemma}, for $k>\lambda_1$, we know $S^\lambda/Q_k=0$. So the summands in the decomposition of (\ref{decomposition}) are stable when $n\geq \lambda_1$. Since $S^\lambda/Q_{\lambda_1}\neq0$, then $T$ is not surjective for $n=\lambda_1-1$. We have completed the proof of the lemma.

\end{proof}
\section{Noetherian property}
In this section we will study the Noetherian property of $FI_\mathscr{H}$-modules. Assume that $FI_\mathscr{H}$-module $V$ is finitely generated, then the degree of $V$ is also finite. Conversely, let $V$ have degree $\leq d$. If in addition $V_n$ for $n\leq d$ is finitely generated $\mathscr{H}_n$-modules, then $V$ is finitely generated.

\begin{definition}
Given $a\geq0$. Let $V=(V_n,\phi_n)$ be an $FI_\mathscr{H}$-module. We define an $FI_\mathscr{H}$-module $S_{+a}V$ by $(S_{+a}V)_n=V_{n+a}$. As $\mathscr{H}_n$-module, $(S_{+a}V)_n$ is isomorphic to $\Res^{\mathscr{H}_{n+a}}_{\mathscr{H}_n}V_{n+a}$.
\end{definition}

\begin{lemma}\label{Sadec}
Let $m\geq1$ and $a\geq0$. Then there exists a decomposition of $S_{+a}M(m)$:
$$S_{+a}M(m)=M(m)\oplus C_a,$$
where $C_a$ is finitely generated in degree $\leq m-1$.
\end{lemma}
\begin{proof}
As we have computed in the previous sections:
\begin{align}
\big(S_{+a}M(m)\big)_n & =M(m)_{n+a}\\
 & =\mathscr{H}_{a+n}\bigotimes_{\mathscr{H}_{\lambda_n}}\mathscr{H}_m\\
 & =\bigoplus_{\sigma\in\mathscr{D}_{\mu_n,\lambda_n}}(\mathscr{H}_{\mu_n}T_{\sigma}\bigotimes_{\mathscr{H}_{\lambda_n}}\mathscr{H}_m),
\end{align}
where $\mu_n=(1,1,\cdots,1,n)$ and $\lambda_n=(m,a+n-m)$.

Let $\sigma\in\mathscr{D}-\{(1)\}$. There exists $m'\in\{1,2,\cdots,m-1\}$ such that $\sigma^{-1}\mathfrak{S}_{\mu_n}\sigma\cap\mathfrak{S}_m=S_{\{m'+1,\cdots,m\}}$.
Then we can decompose $\mathfrak{S}_m=\bigsqcup_{\sigma'\in \mathscr{D}'}S_{\{m'+1,\cdots,m\}}\sigma'\mathfrak{S}_{m'}$, where $\mathscr{D}'$ is the corresponding representative elements of minimal length. So we have:
\begin{align}
\bigoplus_{\sigma\in\mathscr{D}_{\mu_n,\lambda_n}}(\mathscr{H}_{\mu_n}T_{\sigma}\bigotimes_{\mathscr{H}_{\lambda_n}}\mathscr{H}_m)
 &=(\mathscr{H}_{\mu_n}\bigotimes_{\mathscr{H}_{\lambda_n}}\mathscr{H}_m)\bigoplus\Big(\bigoplus_{\sigma\in\mathscr{D}_{\mu_n,\lambda_n}-\{(1)\}}\big(\bigoplus_{\sigma'\in\mathscr{D}'}
(\mathscr{H}_{\mu_n}T_\sigma T_{\sigma'}\bigotimes_{\mathscr{H}_{\lambda_n}}\mathscr{H}_{m'})\big)\Big)\\
 &\cong M(m)_n\oplus C_a,
\end{align}
where $C_a$ is finitely generated in degree $\leq m-1$.
\end{proof}
From the above lemma, we know the degree of $S_{+a}M(m)$ is finite and by lemma\ref{Free}, $M(m)$ can be replaced by an $FI_\mathscr{H}$-module with finite degree. Conversely, according to our definition of $S_{+a}V$, we have

\begin{lemma}
Let $V$ be a consistent sequence. If $S_{+a}V$ is of finite degree, then so is $V$.
\end{lemma}
\begin{theorem}\textbf{(Noetherian Property)}\label{Noetherian property}
Let $V=(V_n,\phi_n)$ be finitely generated and $W\subseteq V$. Then $W$ is also finitely generated.
\end{theorem}
\begin{proof}
Suppose $V$ is generated in degree $m$.
By lemma\ref{Free}, $V$ is a quotient of a finite direct sum of $M(m_i)$ for $m_i\leq m$. So we only need to prove $M(m_i)$ satisfies the Noetherian property. Let $W\subseteq M(m)$. Since the injective degree of $M(m)$ is zero, we can assume that $M(m)_n\subseteq M(m)_{n+1}$ for all $n$. Since every $W_n$ is finite dimensional, then $W$ has finite degree implies $W$ is finitely generated. To show $\deg{(W)}<\infty$ it suffices to prove it for $S_{+a}W$ for some $a$.

By lemma \ref{Sadec}, $S_{+a}M(m)$ can be decomposed as two summand, one of degree $m$. And so is $S_{+a}W$. Since the map $\phi|_{W_n}: W_n\to W_{n+1}$ is injective, we can assume $W_n\subseteq W_{n+1}$. For any $a$, $\mathscr{H}_m=M(m)_m$ is contained in $S_{+a}M(m)_m$, which is exactly the summand of degree $m$. Let $W^a_m$ be the summand of $S_{+a}W_m$. Then $W^a_m\subseteq W^{a+1}_m\subseteq\mathscr{H}_m$. Since $\mathscr{H}_m$ is finite dimensional, there exists $N$ such that $W^N_m=W^{N+1}_m$. For any $a\geq0$, $W^{N+a}_m$ generates $W^N_{m+a}$ which implies that $W^N$ is finitely generated. For the degree $\leq m-1$ of $S_{+N}W$, we use induction which implies $S_{+N}W$ is finitely generated of degree.

\end{proof}

\section{representation stability}
The following lemma shows that the representation stable range $N<\infty$ if the stability degree and weight of an $FI_\mathscr{H}$-module are finite.
\begin{lemma}\label{stable range}
Let $V=(V_n,\phi_n)$ be an $FI_\mathscr{H}$-module with stability degree $s$ and weight $m$. Then $V$ is uniformly representation stable with stable range $n\geq s+m$.
\end{lemma}
\begin{proof}
Let $\phi_n': \mathscr{H}_{n+1}\bigotimes_{\mathscr{H}_n}V_n\to V_{n+1}$.

In order to indicate the injective and surjectivity, that is to say $\ker{\phi_n}=0$ and $\coker{\phi_n'}=0$ for $n\geq s+m$, it suffices to prove $\ker{\phi_n}/Q_{n-m}=0$ and $\coker{\phi_n'}/Q_{n+1-m}$$=0$ by lemma\ref{branching lemma}.

Since $n-m\geq s$, the map $T: \Phi_m(V)_{n-m}\to\Phi_m(V)_{n+1-m}$ is bijective. Moreover, notice that the following diagrams
$$\xymatrix{
                & V_{n+1}/Q_{n-m} \ar[dr]^{T_2}             \\
 V_n/Q_{n-m} \ar[ur]^{T_1} \ar[rr]^{T} & &     V_{n+1}/Q_{n+1-m}        }
$$
and
$$\xymatrix{
                & (\mathscr{H}_{n+1}\bigotimes_{\mathscr{H}_n}V_n)/Q_{n+1-m} \ar[dr]^{T_2'}             \\
 V_n/Q_{n-m} \ar[ur]^{T_1'} \ar[rr]^{T} & &     V_{n+1}/Q_{n+1-m}        }
$$
commute. So $T_1$ is injective and $T_2'$ is surjective. We have obtained that $\ker{\phi_n}/Q_{n-m}=0$ and $\coker{\phi_n'}/Q_{n+1-m}=0$.

Assume that $V_n=\bigoplus_{\lambda} c_{\lambda,n}S(\lambda)_n$, we need to show $c_{\lambda,n}$ is independent of $n$ when $n\geq s+m$. Since the weight of $V$ is $m$, there is no $\lambda$ with $|\lambda|>m$ occurred in $V$. By lemma\ref{branching lemma}, $S(\lambda)_n/Q_{n-m}=0$ when $|\lambda|>m$. So
\begin{align}
V_n/Q_{n-m} &=\bigoplus_{|\lambda|\leq m}c_{\lambda,n}\big(S(\lambda)_n/Q_{n-m}\big)\\
  &=\bigoplus_{|\lambda|<m}c_{\lambda,n}\big(S(\lambda)_n/Q_{n-m}\big)\bigoplus\Big(\bigoplus_{|\lambda|=m}c_{\lambda,n}\big(S(\lambda)_n/Q_{n-m}\big)\Big)\\
  &=\bigoplus_{|\lambda|<m}c_{\lambda,n}\big(S(\lambda)_n/Q_{n-m}\big)\bigoplus(\bigoplus_{|\lambda|=m}c_{\lambda,n}S^{\lambda}).
\end{align}
Because $n-m$ is great or equal to the stability degree and for the first term of the above decomposition applies the induction, we obtain $c_{\lambda,n}$ is independent of $n$.

\end{proof}

Now we state and prove our main theorem about representation stability associated with Hecke algebra.
\begin{theorem}\label{mt}
An $FI_\mathscr{H}$-module $V=(V_n,\phi_n)$ is finitely generated if and only if the sequence $(V_n, \phi_n)$ is uniformly representation stable and each $V_n$ is finite-dimensional.
\end{theorem}
\begin{proof}
Assume $V$ is uniformly stable with range $N$. Because of the surjectivity of $V$, we have $V_n=span(\im\phi_{n-1})_n$ for all $n-1\geq N$. So $span(\bigsqcup_{n=0}^NV_n)=V$ and together with the finite-dimensional imply that $V$ is finitely generated.

For the converse, by proposition \ref{stable range}, we only need to show $V$ has finite stability degree and weight. According to corollary \ref{corollary}, the surjective degree is finite. Assume that $V$ is generated in degree $\leq d$. There exists an epimorphism $g: M\to V$, where $M=M(W)$ for some $W=(W_n)_{n\geq 0}$. Let $K$ be the kernel of $g$. By corollary\ref{Noetherian property}, $K$ is also finitely generated and so finite surjective degree, say $s$.

For given $a\geq 0$. For any $n\geq s$, we have the following commutative diagram:

$$\CD
  0 @>>> \Phi_a(K)_n @>>> \Phi_a(M)_n @>>> \Phi_a(V)_n @>>> 0 \\
   @.  @V T_2 VV @V T_1 VV @V T VV   \\
  0 @>>> \Phi_a(K)_{n+1} @>>> \Phi_a(M)_{n+1} @>>> \Phi_a(V)_{n+1} @>>> 0
\endCD
$$
By the Snake Lemma we have an exact sequence:

$$\xymatrix@C=0.5cm{
  \ker T_1 \ar[rr]^{} && \ker T \ar[rr]^{} && \coker T_2 \ar[rr]^{} && \coker T_1 \ar[rr]^{} && \coker T \ar[r] & 0 }.$$
So $\ker T_1=0$ and $\coker T_2=0$ imply $\ker T=0$. This shows the injective degree of $V$ is finite. We thus have completed the proof.

\end{proof}





\bibliographystyle{amsplain}

\end{document}